\documentclass[11pt]{amsart}

\usepackage{amssymb,amsmath,latexsym}
\usepackage[varg]{pxfonts}
\usepackage{amsmath, amsthm, amscd, amsfonts, amssymb, graphicx, color}
\title[A note on essentially left $\phi$-contractible Banach algebras]{A note on essentially left $\phi$-contractible Banach algebras}

 \author[A. Sahami]{Amir Sahami}
 
 \email{a.sahami@ilam.ac.ir}
 
 \address{Faculty of Basic sciences, Department of Mathematics, Ilam University, P.O.Box 69315-516, Ilam,
 	Iran.}
 \author[I. Almasi]{Isaac Almasi}
 
 \email{i.almasi@ilam.ac.ir}
 
 \address{Faculty of Basic sciences, Department of Mathematics, Ilam University, P.O.Box 69315-516, Ilam,
 	Iran.}
 
\theoremstyle{plain}
\newtheorem{lem}{\textbf{Lemma}}[section]
\newtheorem{thm}[lem]{\textbf{Theorem}}

\theoremstyle{definition}

\theoremstyle{definition}
\theoremstyle{remark}

\theoremstyle{definition}

\theoremstyle{definition}
\theoremstyle{remark}

\begin{document}
\begin{large}

\maketitle

\begin{abstract}
\begin{normalsize}
In this note, we show that   \cite[Corollary 3.2]{sad} is not always true. In fact, we characterize essential  left $\phi$-contractibility of the the group algebras in the term of compactness of its related locally compact group. Also we show that for any compact commutative group $G$, $L^{2}(G)$ is always  essentially left $\phi$-contractible. We discuss essential left $\phi$-contractibility of some Fourier algebras.   
\end{normalsize}
\end{abstract}

\begin{normalsize}
   \end{normalsize}

\section{ \textbf{Introduction and preliminaries }}
Johnson introduced and studied the notion of amenability for Banach algebras. A Banach algebra $A$ is called amenable, if every continuous (bounded linear) derivation $D$ from $A$ into $X^*$  is inner, that is, $D$ has a form $$D(a)=a\cdot x_{0}-x_{0}\cdot a\quad (a\in A),$$ for some $x_{0}\in X^*$, where $X$ is a Banach $A$-bimodule. For the history of amenability of Banach algebras, see \cite{run}.

Ghahramani and Loy in \cite{gen 1} defined a generalized notion of amenability for Banach algebras called  essential amenability, that is, every continuous derivation $D$ from $A$ into $X^*$  is inner, where $X$ is a neo-unital Banach $A$-bimodule ($X=A\cdot X\cdot A$).

Kanuith et. al. in \cite{kan} defined and investigated the notion of left $\phi$-amenability for a Banach algebra $A$, where $\phi$-is a non-zero multiplicative linear functional. Indeed a Banach algebra $A$ is left $\phi$-amenable if every derivation $D:A\rightarrow X^*$ is inner, where $X$ is a Banach $A$-bimodule with the left module action $a\cdot x=\phi(a)x$ for all $a\in A, x\in X$. It is known that for a locally compact group $G$,  the group algebra $L^{1}(G)$ is left $\phi$-amenable if and only if $G$ is amenable. Also the Fourier algebra $A(G)$ is always left $\phi$-amenable, see \cite{san}, \cite{kan}.

Motivated by these considerations Nasr-isfahani et. al. in \cite{nas essent} introduced  the concept of essential left $\phi$-amenability for Banach algebras. A Banach algebra $A$ is called essentially left $\phi$-amenable if every derivation $D:A\rightarrow X^*$ is inner, where $X$ is a neo-unital Banach $A$-bimodule with the left module action $a\cdot x=\phi(a)x$ for all $a\in A, x\in X$.   Nasr-isfahani et. al. studied some Banach algebras related to a locally compact groups under the concept of essential left $\phi$-amenability.

Recently R. Sadeghi Nahrkhalaji defined the concept of essential left $\phi$-contractible for Banach algebras.  A Banach algebra $A$ is called essentially left $\phi$-contractible if every continuous derivation $D:A\rightarrow X^*$ is inner, where $X$ is a neo-unital Banach $A$-bimodule with the right module action $x\cdot a=\phi(a)x$ for all $a\in A, x\in X$, see \cite{sad}. R. Sadeghi Nahrkhalaji studied the  essentially left $\phi$-contractibility of   some Banach algebras related to a locally compact group. Also some hereditary properties of this new notion are given in \cite{sad}.

 In this paper, we study essentially left $\phi$-contractibility of Banach algebras.we show that   \cite[Corollary 3.2]{sad} is not always true. In fact, we characterize essential  left $\phi$-contractibility of the the group algebras in the term of compactness of its related locally compact group. Also we show that for any compact commutative group $G$, $L^{2}(G)$ is always  essentially left $\phi$-contractible. We discuss essential left $\phi$-contractibility of some Fourier algebras.
 
 We give some  notations and definitions that we use
 in this paper frequently. Suppose that  $A$ is a Banach algebra. Throughout this manuscript,
 the character space of $A$ is denoted by $\Delta(A)$, that is, all
 non-zero multiplicative linear functionals (characters)  on $A$.
 
 The projective
 tensor product
 $A\otimes_{p}A$ is a Banach $A$-bimodule via the following actions
 $$a\cdot(b\otimes c)=ab\otimes c,~~~(b\otimes c)\cdot a=b\otimes
 ca\hspace{.5cm}(a, b, c\in A).$$ The product morphism $\pi_{A}:A\otimes_{p}A\rightarrow A$ is given by $\pi_{A}(a\otimes b)=ab,$ for every $a,b\in A.$
 Let $X$ and $Y$ be Banach $A-$bimodules. The map $T:X\rightarrow Y$ is called $A-$bimodule morphism, if
 $$T(a\cdot x)=a\cdot T(x),\quad T(x\cdot a)=T(x)\cdot a,\qquad (a\in A,x\in X).$$

\section{\textbf{Essential left $\phi$-contractibility}}

Note that the Cohen-Hewit factorization is valid, whenever the Banach algebra $A$  has a "bounded" left approximate identity, see \cite[Theorem 1.1.4, p2]{kan lau}. Then in \cite[Proposition 2.3]{sad} to show that $A\otimes_{p}A$ is neo-unital, $A$ must have a bounded approximate identity.
So we state the correct version of \cite[Proposition 2.3]{sad} here.

\begin{thm}
	Let $A$ be a Banach algebra with a bounded approximate identity and $\phi\in\Delta(A).$ Then $A$ is left $\phi$-contractible if and only if $A$ is essentially left $\phi$-contractible.
\end{thm}
\begin{proof}
	See the proof of \cite[Proposition 2.3]{sad}.
\end{proof}

Let $G$ be a locally compact group and $L^{1}(G)$ be its associated group algebra. We denote $\widehat{G}$ for the dual group of $G$, that is, the set of all non-zero continuous homomorphisms $\rho$ from $G$ into $\mathbb{T}=\{z\in\mathbb{C}:|z|=1\}$. It is known that every non-zero multiplicative linear functional on $L^{1}(G)$ has the form $\phi_{\rho}$ for some $\rho\in\widehat{G},$ where $$\phi_{\rho}(f)=\int_{G} \overline{\rho(x)}f(x)dx,\quad f\in L^{1}(G),$$
where $dx$ is denoted for the  Haar measure.
For more information about the characters of  group algebra see \cite[Theorem 23.7]{hew}.

We should remind that a Banach algebra is left $\phi$-contractible if and only if there exists an element $m\in A$ such that $am=\phi(a)m$ and $\phi(m)=1$ for all $a\in A$. For knowing more about left $\phi$-contractibility of a Banach algebra and its hereditary properties through the homological approach, see \cite{Nas1}.
\begin{thm}
Let $G$ be a locally compact group. Then $L^{1}(G)$ is essentially $\phi$-contractible if and only if $G$ is compact.
\end{thm}
\begin{proof}
	Let $G$ be a compact group. Then each continuous homomorphism $\rho:G\rightarrow \mathbb{T}$ belongs to $L^{\infty}(G)$. On the other hand $L^{\infty}(G)\subseteq L^{1}(G)$. So $\rho\in L^{1}(G).$ Consider	
	\begin{equation*}
	\begin{split}
	f\ast \rho(x)=\int_{G} f(y)\rho(y^{-1}x)dy&=\int_{G} f(y)\rho(y^{-1})\rho(x)dy\\
	&=\rho(x)\int_{G} f(y)\overline{\rho(y )}dy\\
	&=\rho(x)\phi_{\rho}(f).
	\end{split}
	\end{equation*}
	It follows that $f\ast \rho=\phi_{\rho}(f)\rho.$  Also
			\begin{equation*}
	\begin{split}
	\phi_{\rho}(\rho)=\int_{G} \rho(x)\overline{\rho(x)}dy&=\int_{G} \rho(x)\rho(x^{-1})dy\\
	&=\int_{G} 1 dx=1,
	\end{split}
	\end{equation*}

	here we consider the normalized Haar measure on $G$. Thus by \cite[Theorem 2.1]{Nas1} $L^{1}(G)$ is left $\phi_{\rho}$-contractible. So $L^{1}(G)$  is essentially left $\phi$-contractible
	
	Conversely, suppose that $L^{1}(G)$ is essentially left $\phi_{\rho}$-contractible. Since $L^{1}(G)$ has a bounded approximate identity, by \cite[Proposition 2.3]{sad}, essential left  $\phi_{\rho}$-contractibility of $L^{1}(G)$ implies the left   $\phi_{\rho}$-contractibility of $L^{1}(G)$. Applying \cite[Theorem 3.3]{ala} follows that $G$ is compact.
\end{proof}

In the following theorem we show that $ii\Rightarrow iv$ of \cite[Corollary 3.2]{sad} is not valid just only for a finite group $G$. 
Suppose that $G$ is a locally compact group. It is well-known that $L^{2}(G)$ is a Banach algebra with convolution if and only if $G$ is compact. 
\begin{thm}	
Let $G$ be a compact commutative  group. Then $L^{2}(G)$ is essentially left $\phi_{\rho}$-contractible, for each $\phi_{\rho}\in\Delta(L^{2}(G))$.	
\end{thm}	
\begin{proof}
Let $L^{2}(G)$ be essentially left $\phi_{\rho}$-contractible. It is known by Plancherel Theorem \cite[Theorem 1.6.1]{rud} that $L^{2}(G)$ is isometrically isomorphic to $\ell^{2}(\widehat{G})$, where $\ell^{2}(\widehat{G})$ is equipped with the pointwise multiplication. Zhang in \cite[Example]{zhang} showed that $\ell^{2}(\widehat{G})$ is approximately biprojective, that is, there exists a net $\rho_{\alpha}:\ell^{2}(\widehat{G})\rightarrow \ell^{2}(\widehat{G})\otimes_{p}\ell^{2}(\widehat{G})$ of $\ell^{2}(\widehat{G})$-bimodule morphisms such that $\pi_{\ell^{2}(\widehat{G})}\circ\rho_{\alpha}(a)\rightarrow a,$ for each $a\in \ell^{2}(\widehat{G})$, where $\pi_{\ell^{2}(\widehat{G})}:\ell^{2}(\widehat{G})\otimes_{p}\ell^{2}(\widehat{G})\rightarrow \ell^{2}(\widehat{G})$ is the product morphism given by $\pi_{\ell^{2}(\widehat{G})}(a\otimes b)=ab$ for all $a,b \in\ell^{2}(\widehat{G}). $
 On the other hand suppose that $\Lambda $  is the collection of all  finite subsets of $\widehat{G}$. Clearly with the inclusion $\Lambda$ becomes an ordered set. One can see that $$\{u_{\beta}=\sum_{i\in \beta}e_{i}:\beta\in \Lambda\},$$here $e_{i}$ is an element of $\ell^2(G)$ equal to $1$ at $i$ and $0 $ elsewhere, forms a central approximate identity for $\ell^{2}(\widehat{G})$. Then for each $\phi_{\rho}\in\Delta(L^{2}(G))$ we can find an element $x_{0}\in L^{2}(G) $ such that $ax_{0}=x_{0}a$ and $\phi_{\rho}(x_{0})=1$ for each $a\in L^{2}(G).$ Applying \cite[Lemma 3.5]{nem homlog} follows that $L^{2}(G)$ is     $\phi_{\rho}$-contractible, for each $\phi_{\rho}\in\Delta(L^{2}(G))$. Therefore $L^{2}(G)$ is   essentially left $\phi_{\rho}$-contractible, for each $\phi_{\rho}\in\Delta(L^{2}(G))$.
\end{proof}
\begin{thm}
Let $G$ be an amenable group and $A(G)$ be the  Fourier algebra on $G$. Then $A(G) $ is essentially left $\phi$-contractible if and only if $G$ is discrete.
\end{thm}
\begin{proof}
	Let $A(G) $ be essentially left $\phi$-contractible. Since $G$ is amenable by Leptin's Theorem (\cite[Theorem 7.1.3]{run}), amenability of $G$ implies that  $A(G)$ has a bounded approximate identity. By \cite[Proposition 2.3]{sad} essentially left $\phi$-contractibility of $A(G)$ gives that $A(G)$ is left $\phi$-contractible. Using \cite[Theorem 3.5]{ala} shows that $G$ is discrete.
	
	Converse is clear by \cite[Theorem 3.5]{ala}
\end{proof}

\end{large}
\end{document}